\documentclass[12pt,leqno]{amsart}
\usepackage{amsmath,amsfonts,amssymb,amsthm,amscd}
\theoremstyle{plain}
\newtheorem{theorem}{Theorem}[section]
\newtheorem{lemma}[theorem]{Lemma}

\theoremstyle{definition}
\newtheorem{defn}{Definition}[section]
\theoremstyle{remark}

\usepackage{mathtools}
\title[Existence of finite global norm of potential vector field]{Existence of finite global norm of potential vector field in a Ricci soliton}
\author[A. A. Shaikh, C. K. Mondal and P. Mandal]{Absos Ali Shaikh$^{1*}$, Chandan Kumar Mondal$^2$ and Prosenjit Mandal$^3$}

\address{\noindent\newline $^{1,2,3}$Department of Mathematics,\newline The University of Burdwan,Golapbag,\newline Purba Bardhaman-713101,\newline West Bengal, India}
\email{$^1$aask2003@yahoo.co.in, aashaikh@math.buruniv.ac.in}
\email{$^2$chan.alge@gmail.com}
\email{$^3$prosenjitmandal235@gmail.com}

\begin{document}
\begin{abstract}
In this article, we investigate global norm of potential vector field in Ricci soliton. In particular, we have deduced certain conditions so that the potential vector field has finite global norm in expanding Ricci soliton. We have also showed that if the potential vector field has finite global norm in complete non-compact Ricci soliton having finite volume, then the scalar curvature becomes constant.  
\end{abstract}

\noindent\footnotetext{ $^*$ Corresponding author.\\ $\mathbf{2010}$\hspace{5pt}Mathematics\; Subject\; Classification: 53C20; 53C21; 53C44. }
\maketitle

\section{Introduction and preliminaries}
A Riemannian manifold $(M,g)$ is called the Ricci soliton \cite{HA82} if its Ricci tensor $Ric$ satisfies the following equation:
\begin{equation}\label{y1}
2Ric+\pounds_Xg=2\lambda g,
\end{equation}
where $\pounds_Xg$ is the lie derivative of the metric tensor $g$ with respect to the vector field $X$. In tensors of local coordinates system, (\ref{y1}) can also be expressed in the following form
\begin{equation}\label{y2}
2R_{ij}+\nabla_iX_j+\nabla_jX_i=2\lambda g_{ij},
\end{equation} 
where $R_{ij}$ denotes the components of the Ricci tensor. The Ricci soliton is said to be expanding (resp., steady and shrinking) if $\lambda<0$ (resp., $\lambda=0$ and $\lambda>0$). The vector field $X$ is known as potential vector field of the Ricci soliton and if $X$ is the gradient of some smooth function $f\in C^\infty(M)$, then $f$ is called the potential function and in this case the Ricci soliton is called the gradient Ricci soliton. After normalization, gradient expanding Ricci soliton reduces to 
\begin{equation}\label{e3}
R_{ij}+\nabla_i\nabla_jf=-\frac{1}{2}g_{ij}.
\end{equation}
\par The study of curvature estimation in Ricci soliton is a well known research area in differential geometry. In 2009, Zhang \cite{ZH09} proved that the scalar curvature is non-negative in case of steady and shrinking gradient Ricci soliton and the scalar curvature is bounded below in case of expanding gradient Ricci soliton. Later, Chow et. al \cite{CLY11} improved this estimation for shrinking gradient Ricci soliton. Petersen and Wylie \cite{PW09} proved that a gradient expanding Ricci soliton with constant scalar curvature $R$ satisfies $-\frac{n}{2}\leq R\leq 0$ and if $R=-\frac{n}{2}$, the metric becomes Einstein, and also Pigola-Rimoldi-Setti \cite{PRS11} and Zhang \cite{ZH11} independently showed that for such a soliton the scalar curvature $R\geq-\frac{n}{2}$. Again,  Mondal and Shaikh \cite{CA20} studied gradient Ricci soliton having non-negative Ricci curvature and realizing convex potential which fulfills finite weighted Dirichlet condition and proved that in such a case scalar curvature vanishes. Recently, Chen \cite{CH20} proved that if the Ricci curvature in a gradient expanding Ricci soliton satisfies $$\lim\limits_{l(x)\rightarrow\infty}l(x)^2|Ric|=0,$$
where $l(x)$ is the distance of $x$ from a fixed point, then the scalar curvature is non-negative, and also showed that if the Ricci curvature in a gradient expanding Ricci soliton is non-positive, then the scalar curvature is negative. In the present paper, we have deduced certain conditions so that the potential vector field has finite global norm in expanding Ricci soliton. We have also showed that if the potential vector field has finite global norm in complete non-compact Ricci soliton having finite volume, then the scalar curvature becomes constant.  Throughout this article, by the notation $(M,g)$ or $M$ we denote the Riemannian manifold $M$ with the Riemannian metric $g$.
\par For any $k$, $0\leq k\leq n$, let $\Lambda^k(M)$ be the vector space of all smooth $k$-forms in $M$. For any $\omega,\eta\in \Lambda^k(M)$, the local inner product of $\omega$ and $\eta$ is denoted by $(\omega,\eta)$ and is defined by, see \cite[p. 149]{MO01},
$$(\omega,\eta)=\omega_{i_1,\cdots,i_k}\eta^{i_1,\cdots,i_k},$$
where $\omega=\omega_{i_1,\cdots,i_k}dx^{i_1}\wedge\cdots\wedge dx^{i_k}$, $\eta=\eta_{i_1,\cdots,i_k}dx^{i_1}\wedge\cdots\wedge dx^{i_k}$ and $\eta^{i_1,\cdots,i_k}=g^{i_1j_1}\cdots g^{i_kj_k}\eta_{i_1,\cdots,i_k}$. For a fixed $k\geq 0$, the Hodge star operator $*:\Lambda^k(M)\rightarrow\Lambda^{n-k}(M)$ is defined by
$$*\omega=sgn(I,J)\omega_{i_1,\cdots,i_k}dx^{j_1}\wedge\cdots\wedge dx^{j_{n-k}},$$
for $\omega=\omega_{i_1,\cdots,i_k}dx^{i_1}\wedge\cdots\wedge dx^{i_k}\in\Lambda^k(M)$. Here $j_1<\cdots<j_{n-k}$ is the rearrangement of the complements of $i_1<\cdots <i_k$ in the set $\{1,...,n\}$ in ascending order and $sgn(I,J)$ is the sign of the permutation $i_1,...,i_k,j_1,...,j_{n-k}$. For an oriented Riemannian manifold $M$, the global inner product in $\Lambda^k(M)$ is defined by
$$\langle \omega,\eta\rangle=\int_M \omega\wedge *\eta,$$
for $\omega,\eta\in\Lambda^k(M)$. We define the global norm of $\omega\in\Lambda^k(M)$ by $\|\omega\|^2=\langle \omega,\omega\rangle$ and remark that $\|\omega\|^2\leq \infty$. There is a natural adjoint operator of the exterior derivative $d:\Lambda^k(M)\rightarrow\Lambda^{k+1}(M)$ called the co-differential operator $\delta:\Lambda^k(M)\rightarrow\Lambda^{k-1}(M)$ and is defined by
$$\delta=(-1)^k*^{-1}d*=(-1)^{n(k+1)+1}*d*,$$
so that the following diagram commutes\\
\begin{center}
$\begin{CD}
\Lambda^k(M) @>*>> \Lambda^{n-k}(M)\\
@VV\delta V @VV d V\\
\Lambda^{k-1}(M) @>(-1)^k*>> \Lambda^{n-k+1}(M)
\end{CD}$
\end{center}
\vspace{.6 cm}
For a $1$-form $\omega\in\Lambda^1(M)$, we have
\begin{eqnarray*}
(d\omega)_{ij}&=&\nabla_i\omega_j-\nabla_j\omega_i \quad \text{ and }\\
(\delta\omega)&=&-\nabla^i\omega_i,
\end{eqnarray*}
where $\nabla^i=g^{ij}\nabla_j$. For a detailed discussion on Hodge operator and co-differential operator see \cite{MO01}. By the notation $\Lambda^k_0(M)$, we denote the subspace of $\Lambda^k(M)$ containing all $k$-forms in $M$ with compact support. Also the completion of $\Lambda^k_0(M)$ with respect to the global inner product $\langle,\rangle$ is denoted by $L^s_2(M)$.
\par In this article, a vector field and its dual $1$-form with respect to the Riemannian metric $g$ will be denoted by the same letter, i.e., for a vector field $X=X^i\partial_i$, the dual $1$-form is $X=X_idx^i=g_{ij}X^jdx^i$.
\section{Main results}
\begin{defn}\cite{YO84}
 A vector field on a manofold $M$ is said to have finite global norm  if its dual $1$-form with respect to $g$ belongs to $L^1_2(M)\cap \Lambda^1(M)$.
\end{defn}
 \par Let $o\in M$ be a fixed point and $l(x)$ be the distance from $o$ to $x$ for each $x\in M$. The open ball with center $o$ and radius $r>0$ is denoted by $B(r)$. Then there exists a Lipschitz continuous function $\omega_r$ such that for some constant $K>0$, (see \cite{YA76}),
\begin{eqnarray*}
&&|d\omega_r|\leq \frac{K}{r}\qquad \text{almost everywhere on }M\\
&&0\leq \omega_r(x)\leq 1\quad\forall x\in M\\
&&\omega_r(x)=1\quad\forall x\in B(r)\\
&& \text{ supp }\omega_r\subset B(r)
\end{eqnarray*}
Then taking limit, we get $\lim\limits_{r\rightarrow\infty}\omega_r=1$. Now we calculate the global norm of $\omega_rdX$ and $\omega_r\delta X$ when $R_{ij}(\nabla^iX^j)=-\lambda R$. Then we have
\begin{eqnarray}\label{l1}
g(dX,dX)&=& 
\nonumber \frac{1}{2}\{(\nabla_iX_j-\nabla_jX_i)(\nabla^iX^j-\nabla^jX^i) \}\\
\nonumber&=&\frac{1}{2}\{4(\nabla_iX_j)(\nabla^iX^j)+4R_{ij}(\nabla^iX^j)-4\lambda \nabla^iX^i \}\\
\nonumber&=& 4g(\nabla X,\nabla X)+2 R_{ij}(\nabla^iX^j)-2\lambda(\lambda n-R)\\
&=& 4g(\nabla X,\nabla X)-2n\lambda^2,
\end{eqnarray}

\begin{eqnarray}\label{l2}
\nonumber g(\delta X,\delta X)&=&(\nabla^iX_i)(\nabla^jX_j)\\
&=&g(\lambda n-R,\lambda n-R)=|\lambda n-R|^2.
\end{eqnarray}
The above two equations (\ref{l1}) and (\ref{l2}) together imply the following lemma:
\begin{lemma}\label{L1}
If $R_{ij}(\nabla^iX^j)=-\lambda R$, then the potential vector field $X$ of the Ricci soliton (\ref{y2}) satisfies the following:
\begin{equation}\label{l3}
\|\omega_rdX\|^2_{B(2r)}=4\|\omega_r\nabla X\|^2_{B(2r)}-2n \|\omega_r\lambda\|^2_{B(2r)},
\end{equation}
\begin{equation}\label{l4}
\|\omega_r\delta X\|^2_{B(2r)}=\|\omega_r(\lambda n-R)\|^2_{B(2r)}.
\end{equation}
\end{lemma}
Combining lemma 2 and lemma 3 of \cite{YO84}, we have
\begin{lemma}\cite{YO84}\label{L3}
For any $X\in \Lambda^1(M)$, we have
\begin{equation}\label{l5}
4\langle \omega_r d\omega_r \otimes X,\nabla X\rangle_{B(2r)}+\langle \omega_r \nabla^2X,\omega_r X\rangle_{B(2r)}+2\langle \omega_r\nabla X,\omega_r\nabla X\rangle _{B(2r)}=0.
\end{equation}
\begin{eqnarray}
\nonumber &&\langle \omega_r \Re X,\omega_r X\rangle_{B(2r)}= \langle \omega_r\nabla^2 X,\omega_r X\rangle_{B(2r)}+\langle \omega_rdX,\omega_rdX\rangle_{B(2r)}\\
&&+2\langle \omega_rdX,d\omega_r\wedge X\rangle_{B(2r)}+\langle \omega_r\delta X,\omega_r\delta X\rangle_{B(2r)}-2\langle \omega_r\delta X,*(d\omega_r\wedge*X)\rangle_{B(2r)},
\end{eqnarray}
where $(\nabla^2 X)_i=\nabla^j\nabla_jX_i$, $(\nabla X)_{ij}=\nabla_iX_j$ and $(\Re X)_i=R^j_iX_j$ is the Ricci transformation on $\Lambda^1(M)$.
\end{lemma}
\begin{lemma}\cite{AN65}\label{L2}
For any $X\in \Lambda^k(M)$, there exists a positive constant $A$ independent of $r$ such that
\begin{eqnarray*}
\|d\omega_r\otimes X\|^2_{B(2r)}&\leq &\frac{A}{r^2}\|X\|^2_{B(2r)},\\
\|d\omega_r\wedge X\|^2_{B(2r)}&\leq &\frac{A}{r^2}\|X\|^2_{B(2r)}\\
 \text{ and }\|d\omega_r\wedge * X\|^2_{B(2r)}&\leq &\frac{A}{r^2}\|X\|^2_{B(2r)}.
\end{eqnarray*}
\end{lemma}
From Lemma \ref{L1} and Lemma \ref{L2}, we have the following:
\begin{eqnarray}
\nonumber |2\langle \omega_r dX,d\omega_r\wedge X\rangle_{B(2r)}|&\leq & 2\|\omega_r dX\|_{B(2r)}\|d\omega_r\wedge X\|_{B(2r)}\\
\nonumber&\leq & \frac{1}{4}\|\omega_r dX\|^2_{B(2r)}+4\|d\omega_r\wedge X\|^2_{B(2r)}\\
\nonumber&\leq & \|\omega_r\nabla X\|^2_{B(2r)}-\frac{n}{2}\|\omega_r\lambda\|^2_{B(2r)}+\frac{4A}{r^2}\|X\|_{B(2r)}^2,
\end{eqnarray}
\begin{eqnarray}
\nonumber|2\langle \omega_r \delta X,*(d\omega_r\wedge*X)\rangle_{B(2r)}|&\leq & 2\|\omega_r \delta X\|_{B(2r)}\|d\omega_r\wedge *X\|_{B(2r)}\\
\nonumber&\leq & \frac{1}{5}\|\omega_r \delta X\|^2_{B(2r)}+5\|d\omega_r\wedge *X\|^2_{B(2r)}\\
&\leq & \frac{1}{5}\|\omega_r(\lambda n-R)\|^2_{B(2r)}+\frac{5A}{r^2}\|X\|_{B(2r)}^2.
\end{eqnarray}
Thus using Lemma \ref{L3}, we calculate for Ricci soliton (\ref{y2}):
\begin{eqnarray}\label{l6}
\nonumber\langle \omega_r\Re X,\omega_r X\rangle_{B(2r)}&=&-4\langle \omega_r d\omega_r\otimes X,\nabla X\rangle _{B(2r)}-2\langle \omega_r\nabla X,\omega_r\nabla X\rangle_{B(2r)}\\
\nonumber&&+\langle \omega_r dX,\omega_r dX\rangle_{B(2r)}+2\langle \omega_r dX,d\omega_r\wedge X\rangle_{B(2r)}\\
\nonumber&&+\langle \omega_r\delta X,\omega_r\delta X\rangle_{B(2r)}-2\langle\omega_r\delta X,*(d\omega_r\wedge *X)\rangle_{B(2r)}\\
\nonumber&\geq & -\frac{1}{2}\|\omega_r\nabla X\|^2 _{B(2r)}-\frac{8A}{r^2}\|X\|^2 _{B(2r)}-2\|\omega_r\nabla X\|^2 _{B(2r)}+4\|\omega_r\nabla X\|^2 _{B(2r)}\\
\nonumber&&-2n \|\omega_r\lambda\|^2_{B(2r)}-\|\omega_r\nabla X\|^2 _{B(2r)}+\frac{n}{2}\|\omega_r\lambda\|^2_{B(2r)}-\frac{4A}{r^2}\|X\|_{B(2r)}^2\\
\nonumber&&+\|\omega_r(\lambda n-R)\|^2_{B(2r)}-\frac{1}{5}\|\omega_r(\lambda n-R)\|^2_{B(2r)}-\frac{5A}{r^2}\|X\|^2 _{B(2r)}\\
&=& \frac{1}{2}\|\omega_r\nabla X\|^2 _{B(2r)}-\frac{17A}{r^2}\|X\|^2 _{B(2r)}-\frac{3n}{2}\|\omega_r\lambda\|^2 _{B(2r)}+ \frac{4}{5}\|\omega_r(\lambda n-R)\|^2_{B(2r)}.
\end{eqnarray}
If the Ricci curvature is non-positive then we can write
$$\limsup_{r\rightarrow\infty}\langle \omega_r\Re X,\omega_r X\rangle_{B(2r)}\leq 0.$$
Hence (\ref{l6}) reduces to the following inequality:
\begin{equation*}\label{l7}
 5\|\nabla X\|^2-15 n\|\lambda\|^2+8\|(\lambda n-R)\|^2\leq 0,
\end{equation*}
which yields
$$5\|\nabla X\|^2+\int_M(-15n\lambda^2+8\lambda^2n^2-16\lambda nR+8R^2)dv\leq 0.$$
Therefore, for $n\geq 2$, the above inequality gives
\begin{equation}\label{e2}
5\|\nabla X\|^2+\int_M(-16\lambda nR+8R^2)dv\leq 0.
\end{equation}
Hence  we get
\begin{equation}\label{e1}
5\|\nabla X\|^2\leq 16\lambda n\int_MR.
\end{equation}
\begin{theorem}
Let $(M,g,X)$ be an expanding Ricci soliton which is complete and non-compact. If the Ricci curvature is non-positive and the following conditions hold:\\
(i) $R_{ij}(\nabla^iX^j)=R/2,$ and\\
(ii) $\int_MR\geq -c$ for some positive constant $c$,
then the potential vector field has finite global norm. Moreover, if the manifold has finite volume, then
$$\int_MR^2\leq ncVol(M).$$
\end{theorem}
\begin{proof}
Since $(M,g,X)$ is a complete non-compact expanding Ricci soliton, we can take $\lambda=-\frac{1}{2}$. Now putting the value of $\lambda$ in (\ref{e1}) and using the given conditions, we get our first result. Also, (\ref{e2}) implies that
$$\int_MR^2\leq -n\int_MR\leq nc Vol(M).$$
\end{proof}
\begin{theorem}
Let $(M,g,X)$ be a complete non-compact Ricci soliton with finite volume. If the potential vector field $X$ is of finite global norm, then the scalar curvature $R$ must be constant. 
\end{theorem}
\begin{proof}
For any $r>0$, we have
\begin{eqnarray}
\nonumber\frac{1}{r}\int_{B(2r)}|X|dV & \leq & \Big(\int_{B(2r)}\langle X,X\rangle dV \Big)^{1/2} \Big(\int_{B(2r)}\Big(\frac{1}{r} \Big)^2 dV\Big)^{1/2}\\
\nonumber&\leq& \|X\|_{B(2r)}\frac{1}{r}\Big(Vol(M)\Big)^{1/2},
\end{eqnarray}
where $Vol(M)$ denotes the volume of $M$. Thus we obtain
\begin{equation*}
\liminf_{\ r\rightarrow \infty} \frac{1}{r}\int_{B(2r)}|X|dV=0.
\end{equation*}
Again we have
\begin{equation*}
\Big|\int_{B(2r)}\omega_r div X dV \Big|\leq \frac{C}{r}\int_{B(2r)}|X|dV,
\end{equation*}
for some constant C. In view of $(\ref{y1})$ the last relation yields
\begin{equation*}
\int_M (n\lambda-R)dV=0,
\end{equation*}
which gives $R=n\lambda$. This completes the proof.
\end{proof}

\section*{Acknowledgment}
 The third author gratefully acknowledges to the
 CSIR(File No.:09/025(0282)/2019-EMR-I), Govt. of India for financial assistance.

\end{document}